\documentclass[11pt]{amsart}
\usepackage[margin=1in]{geometry}

\usepackage{amssymb}
\usepackage{amsthm}
\usepackage{amsmath}
\usepackage{mathrsfs}
\usepackage{amsbsy}
\usepackage{hyperref}
\usepackage{enumerate}
\allowdisplaybreaks
\usepackage{cite}
\usepackage{ragged2e}
\usepackage{graphicx}

\usepackage{comment}

\makeatletter
\newtheorem*{rep@theorem}{\rep@title}
\newcommand{\newreptheorem}[2]{%
\newenvironment{rep#1}[1]{%
 \def\rep@title{#2 \ref{##1}}%
 \begin{rep@theorem}}%
 {\end{rep@theorem}}}
\makeatother

\newtheorem{theorem}{Theorem}
\newreptheorem{theorem}{Conjecture}
\newtheorem{lemma}{Lemma}
\newtheorem{proposition}{Proposition}
\newtheorem{corollary}{Corollary}

\newtheorem{observation}{Observation}
\newtheorem{conjecture}{Conjecture}
\newtheorem{question}{Question}

\theoremstyle{definition}
\newtheorem{definition}{Definition}
\newtheorem{remark}{Remark}

\newcommand{\nnn}{\mathbb{N}}
\newcommand{\zzz}{\mathbb{Z}}
\newcommand{\Z}{\mathbb{Z}}
\newcommand{\N}{\mathbb{N}}

\DeclareMathOperator{\cut} { \setminus}

\begin{document}

\title[]{Sets Arising as Minimal Additive Complements in the Integers} \keywords{}
\subjclass[2010]{}

\author[]{Amanda Burcroff}
\address[]{Centre for Mathematical Sciences, University of Cambridge, Cambridge, UK}
\email{agb63@cam.ac.uk}
\author[]{Noah Luntzlara}
\address[]{Department of Mathematics, University of Michigan, Ann Arbor, MI, USA}
\email{nluntzla@umich.edu}

\begin{abstract} 
 A subset $C$ of an abelian group $G$ is a {\it minimal additive complement} to $W \subseteq G$ if $C + W = G$ and if $C' + W \neq G$ for any proper subset $C' \subset C$.  In this paper, we study which sets of integers arise as minimal additive complements. We confirm a conjecture of Kwon, showing that bounded-below sets with arbitrarily large gaps arise as minimal additive complements. Moreover, our construction shows that any such set belongs to a co-minimal pair, strengthening a result of Biswas and Saha for lacunary sequences.  We bound the upper and lower Banach density of syndetic sets that arise as minimal additive complements to finite sets. We provide some necessary conditions for an eventually periodic set to arise as a minimal additive complement and demonstrate that these necessary conditions are also sufficient for certain classes of eventually periodic sets.   We conclude with several conjectures and questions concerning the structure of minimal additive complements.
\end{abstract}
\maketitle

\section{Introduction}\label{sec: introduction}

Let $C$ and $W$ be subsets of an abelian group $G$, which we will mainly take to be the integers
under addition.
The \emph{Minkowski sum} of $C$ and $W$ is the set $C + W = \{c + w: c \in C, w \in W\}$.  The set $C$ is an \emph{additive complement} to $W$ if $C + W = G$. Moreover, $C$ is a \emph{minimal additive complement} to $W$ if no proper subset $C'\subset C$ is an additive complement to $W$.  In this case, we say $W$ {\it has} a minimal additive complement and $C$ {\it arises as} a minimal additive complement.  In this paper, we study the conditions under which a subset of the integers arises as a minimal additive complement.  

Minimal additive complements were introduced by Nathanson \cite{Nat} in 2011 as an arithmetic analogue to the metric concept of \emph{$h$-nets} in groups. In particular, Nathanson asked which subsets of the integers have a minimal additive complement.  Nathanson showed that every finite set of integers has a minimal additive complement, and moreover that every additive complement to a finite set contains a minimal additive complement.  Chen and Yang 
studied Nathanson's question for infinite subsets of the integers, providing a partial answer with the following results.
\begingroup
\renewcommand*{\thetheorem}{\Alph{theorem}}
\begin{theorem}[Chen, Yang \cite{CY}]\label{thm: CY1}
If $W \subseteq \zzz$ is bounded neither above nor below, then $W$ has a minimal additive complement.
\end{theorem}

\begin{theorem}[Chen, Yang \cite{CY}]\label{thm: CY2}
Let $W = \{w_1 < w_2 < \cdots \}$ be a set of integers and 
$$\overline{W} = \nnn \cut W = \{\overline{w_1} < \overline{w_2} < \cdots\}\,.$$
\begin{enumerate}
    \item[(i)] If $\limsup\limits_{i \to \infty} w_{i+1} - w_i = \infty$, then $W$ has a minimal additive complement.
    \item[(ii)] If $\lim\limits_{i \to \infty} \overline{w_{i + 1}} - \overline{w_i} = \infty$, then $W$ does not have a minimal additive complement.
\end{enumerate}
\end{theorem}
\endgroup
\setcounter{theorem}{0}

Kwon \cite{Kwo} initiated the study of the dual analogue
to Nathanson's question, asking which sets of integers arise as minimal additive complements.  Kwon showed that finite sets of integers arise as minimal additive complements and suggested looking at sets of the types that Chen and Yang investigated in \cite{CY}.  While there are simple counterexamples to the exact dual analogue of Theorem \ref{thm: CY1}, e.g., the nontrivial cofinite sets of integers are bounded neither above nor below and yet do not arise as minimal additive complements, Kwon \cite{Kwo} conjectured that the dual analogue to Theorem \ref{thm: CY2} holds.  We confirm this conjecture
with
the following theorem.  

\begin{theorem}\label{thm: dualCY}
Let $C = \{c_1 < c_2 < \cdots\}$ be an infinite set of natural numbers and 
$$\overline{C} = \nnn \cut C = \{\overline{c_1} < \overline{c_2} < \cdots\}\,.$$
\begin{enumerate}
    \item[(i)] If $\limsup\limits_{i \to \infty} c_{i+1} - c_i = \infty$, then $C$ arises as a minimal additive complement.
    \item[(ii)] If $\lim\limits_{i \to \infty} \overline{c_{i + 1}} - \overline{c_i} = \infty$, then $C$ does not arise as a minimal additive complement.
\end{enumerate}
\end{theorem}

A pair of sets $A,B\subseteq G$ is a \emph{co-minimal pair} if each set is a minimal additive complement to the other.  Observe that a set belonging to a co-minimal pair both has and arises as a minimal additive complement.  
Biswas and Saha \cite{BS2} recently studied lacunary sequences that belong to a co-minimal pair in the integers.  A {\it lacunary sequence} is a subset of positive integers $\{x_n\}_{n \in \N}$ such that $\frac{x_{n + 1}}{x_n} \geq \lambda $ for some real constant $\lambda > 1$ and all $n \in \N$.  Biswas and Saha showed that ``a majority'' of lacunary sequences arise as minimal additive complements (see Section \ref{sec: aperiodic} for relevant discussion).  Motivated by our proof of Kwon's conjecture, we show in Theorem \ref{thm: co-minimal strengthening} that sets satisfying Condition (i) of Theorem \ref{thm: dualCY} belong to a co-minimal pair.  Since lacunary sequences satisfy this condition, we obtain the following corollary, which completes the classification of lacunary sequences belonging to co-minimal pairs.

\begin{corollary}\label{cor: lacunary}
Every lacunary sequence belongs to a co-minimal pair in the integers.
\end{corollary}

In Section \ref{sec: finite}, we bound the upper and lower Banach density of \emph{syndetic sets}, i.e., the sets of integers that have nonempty intersection with every interval of length $k$ for some $k \in \N$, arising as minimal additive complements to finite sets.  We then use these results to study \emph{eventually periodic} sets, as introduced by Kiss, S{\'a}ndor, and Yang in \cite{KSY}, arising as minimal additive complements.  In Section \ref{sec: eventually periodic}, we provide several necessary conditions (which can be checked in finitely many steps) for eventually periodic sets to arise as minimal additive complements.  We furthermore show that these conditions are sufficient for certain families of eventually periodic sets of upper Banach density $(m+1)/3$ with prime period $m \equiv 2 \mod 3$ to arise as minimal additive complements.

Theorems \ref{thm: CY1} and \ref{thm: CY2} show that the only subsets of integers that do not have a minimal additive complement are infinite subsets that are bounded below (or above) and have bounded gaps between consecutive elements. Such sets can be viewed as syndetic sets in $\N$.  Our results show that the subsets of integers that do not arise as minimal additive complements can look qualitatively more diverse. There is still a lot of room for progress toward answering both Nathanson's original question and Kwon's dual analogue; we conclude with several open questions and conjectures in Section \ref{sec: further}.

\section{Preliminaries}\label{sec: preliminaries}

Here we briefly discuss a few basic definitions and results relevant to the rest of the paper. Observe that arising as a minimal additive complement is a translation-invariant and reflection-invariant property of a set of integers.  Thus, to determine which sets of integers that are bounded above or below arise as minimal additive complements, it is enough to consider subsets of $\N$.  

\begin{definition}\label{dependenteltdef}
Let $C$ be an additive complement to $W$.  We say $z \in \zzz$ is {\it dependent} on $c \in C$ if $z \not\in (C \cut \{c\}) + W$.  In this case we say $c$ {\it has} a dependent element (with respect to the sum $C + W$).  
\end{definition}

Observe that $C$ is a minimal additive complement to $W$ if and only if $C$ is an additive complement to $W$ and furthermore every $c \in C$ has a dependent element.  

Given $A,B \subseteq \Z$, we say $z \in A + B$ is {\it finitely $(A,B)$-represented} if the set 
$$\{(a,b) \in A \times B : a + b = z\}$$
is finite.  Biswas and Saha showed that if a pair of sets $A,B \subseteq \Z$ satisfies the property that every $z \in \Z$ is finitely $(A,B)$-represented, then $A$ and $B$ can be reduced to a co-minimal pair.  While we present their result in the context of the integers, they worked a more general setting of countable subsets in any group.

\begin{lemma}[Biswas, Saha \cite{BS2}]\label{lem: finite rep}
Suppose $A,B \subseteq \Z$ are complements and every $z \in \Z$ is finitely $(A,B)$-represented.  Then there exists $A' \subseteq A$ such that $A'$ is a minimal additive complement to $B$. Furthermore, there exists $B' \subseteq B$ such that $(A',B')$ form a co-minimal pair.
\end{lemma}

We often switch between studying subsets of integers and their projections to finite cyclic groups. For a set $C \subseteq \Z$ and $z \in Z$, we denote by $C + z$ the Minkowski sum of $C$ and the singleton $\{z\}$. We denote the standard projection of $C$ to $\Z/m\Z$ by $C_{(m)}$.  We view $C_{(m)} + m\Z$ as a subset of $\Z$ by taking the inverse image of $C_{(m)}$ under the standard projection.

\section{Sets with Large Gaps or Intervals}\label{sec: aperiodic}
In this section, we investigate when sets of natural numbers containing arbitrarily large gaps or having a specified interval structure arise as minimal additive complements in the integers.  Chen and Yang \cite{CY} made partial progress toward determining when such sets have minimal additive complements; their results are summarized in Theorem \ref{thm: CY2}. We prove the dual analogue to their result, Theorem \ref{thm: dualCY}, settling a conjecture of Kwon \cite{Kwo}.  While the statement of Theorem \ref{thm: dualCY} is precisely the dual to that of Theorem \ref{thm: CY2}, the proof does not follow analogously.  In fact, we prove two stronger results, as detailed in Theorems \ref{thm: co-minimal strengthening} and \ref{thm: long runs not MAC}.  

\begin{theorem}\label{thm: co-minimal strengthening}
Let $C = \{c_1 < c_2 < \cdots\}$ be an infinite set of natural numbers.  If $\limsup\limits_{i \to \infty} c_{i+1} - c_i = \infty$, then $C$ belongs to a co-minimal pair in the integers.
\end{theorem}

\begin{theorem}\label{thm: long runs not MAC}
Let $C$ be a set of integers such that $C_+ = C\cap \N = \{c_1 < c_2 < \cdots\}$ is infinite,  and let
$$\overline{C_+} = \nnn \cut C = \{\overline{c_1} < \overline{c_2} < \cdots\}\,.$$
If there exists $\ell > 0$ such that for all $m \in \nnn$ we have $\overline{c_{j}} - \overline{c_i} \notin [\ell,m]$ for sufficiently large $i$ and $j > i$, then $C$ does not arise as a minimal additive complement.
\end{theorem}

Theorem \ref{thm: co-minimal strengthening} allows us to strengthen a recent result of Biswas and Saha \cite[Theorem A]{BS2}, stating that ``a majority'' of lacunary sequences belong to a co-minimal pair in $\Z$.  By ``a majority'', Biswas and Saha meant that their work handled lacunary sequences where $\lambda \geq 6$ and certain lacunary sequences where $\lambda > 3$.  We complete this classification by showing that all lacunary sequences arise as minimal additive complements (see Corollary \ref{cor: lacunary}).  Note that this result follows directly from Theorem \ref{thm: co-minimal strengthening}.

While Condition (ii) of Theorem \ref{thm: dualCY} does not allow gaps of length greater than $1$ to occur infinitely often, Theorem \ref{thm: long runs not MAC} considers sets with bounded gaps and containing increasingly long intervals.  However, we expect that the methods of Chen and Yang could be extended to show that sets of natural numbers satisfying the conditions of Theorem \ref{thm: long runs not MAC} do not have minimal additive complements, thus strengthening Theorem \ref{thm: CY2} (see Conjecture \ref{cor: extendedCY2}).

We now proceed to prove Theorems \ref{thm: dualCY}, \ref{thm: co-minimal strengthening}, and \ref{thm: long runs not MAC} along with several intermediate lemmas.  We begin by showing that for a set $C \subseteq \nnn$ satisfying Condition $(i)$ of Theorem \ref{thm: dualCY}, there exists $D \subseteq \nnn$ such that $C + D$ has arbitrarily large gaps, $C + D$ contains all the non-positive integers, and no integer can be represented as a sum in $C + D$ in infinitely many ways.

It is useful to have a measure of the largest size of a gap in $S$.  To this end, we define the function $g: \mathcal{P}(\Z) \to \N \cup \{\infty\}$.
\begin{definition}
Let $S$ be a set of integers of size at least $2$. We define $g(S) \in \N \cup \{\infty\}$ to be the supremum of the differences between consecutive elements of $S$. 
\end{definition}

\begin{lemma}\label{lem: filling negatives}
Let $C = \{c_1 < c_2 < \cdots\}$ be an infinite set of natural numbers.  If $g(C) = \infty$, then there exists a set $D \subseteq \zzz_{< 0}$ such that $C + D \supseteq \zzz_{\leq 0}$, $g(C + D) = \infty$, and every integer is finitely $(C,D)$-represented.
\end{lemma}
\begin{proof}
We proceed by constructing a sequence $\{D_i\}_{i \geq 0}$ of nested sets  satisfying $C + D_i \supseteq [-i,0]$.  Let $D_0 = \{-c_1\}$.  For $i \in \nnn$, define $y_i$ to be the maximum element of $\zzz_{\leq 0} \cut (C + D_{i-1})$.  Note that, by our inductive hypothesis, we have $y_i \leq -i$.  Define a function $h: \zzz_{\geq 0} \to \zzz_{\geq 0}$ by setting $h(0) = 0$, and for $i \in \nnn$ choose $h(i)$ to be the minimum natural number satisfying
\begin{enumerate}[(a)]
    \item $c_{h(i) + 1} - c_{h(i)} \geq i$,
    \item $g\left([0,c_{h(i)+1}) \cap (C + D_{i-1})\right) \geq i$, and
    \item $h(i) > h(i-1)$.
\end{enumerate}
Such an $h(i)$ must exist because $g(C + D_{i-1}) = \infty$, as $D_{i-1}$ is a finite set.  Let $t_i = c_{h(i) + 1} - y_i$ and $x_i = y_i - c_{h(t_i)}$.  We define $D_i = D_{i-1} \cup \{x_i\}$ and $D = \bigcup_{i \geq 0} D_i$.  Clearly, $D_i \supset D_{i-1}$.  Since $y_i \in C + D_i$, we have $C + D_i \supset (C + D_{i-1}) \cup \{y_i\} \supseteq [-i,0]$.  Thus $C + D \supset \zzz_{\leq 0}$.

Next we show that $g(C + D) = \infty$.   By construction, we have
\begin{align*}
   \min\{z \in C + x_i: z > y_i\} = x_i + c_{h(t_i) + 1} = y_i + c_{h(t_i) + 1} - c_{h(t_i)} \geq y_i + t_i = c_{h(i)+1} \,.
\end{align*}  
Since $h$ is an increasing function, for $i \leq j$ we have
$$\min\{z \in C + x_j: z > y_j\} \geq \min\{z \in C + x_j: z > y_i\} \geq c_{h(j) + 1} \geq c_{h(i) + 1}\,.$$
Hence, $[0,c_{h(i) + 1}) \cap (C + x_j) =\emptyset$ for $j \geq i$.  Combining this with assumption (2) above, we conclude that $g([0,c_{h(i) + 1})\cap (C + D))\geq i$ for all $i \in \nnn$.  Therefore, $g(C + D) = \infty$.

It remains to show that each $w \in C + D$ is finitely $(C,D)$-represented.  Suppose $w < 0$.  As mentioned above, we have $\min\{z \in C + x_i: z > y_i\} = c_{h(i) + 1}$.  Since $y_i \leq -i$, this implies $w \notin C + x_i$ for $i > -w$.  Hence all negative integers are finitely $(C,D)$-represented.  

Now suppose $w \geq 0$.  Note that $h(i) \geq i$ by Condition (c) above.  Thus, 
$$\min\{z \in C + x_i: z > 0\} \geq \min\{z \in C + x_i: z > y_i\} \geq c_{h(i) + 1} \geq c_{i+1} > i\,.$$
Therefore $w \notin C + x_i$ for $i > w$, so $w$ is finitely $(C,D)$-represented.
\end{proof}

Using Lemma \ref{lem: filling negatives}, we now show that any set of natural numbers containing arbitrarily large gaps between consecutive elements arises as a minimal additive complement.

\begin{lemma}\label{lem: large gaps MAC}
Let $C = \{c_1 < c_2 < \cdots\}$ be an infinite set of natural numbers.  If $g(C) = \infty$, then $C$ arises as a minimal additive complement to a set $W$ such that every integer is finitely $(C,W)$-represented.
\end{lemma}
\begin{proof}
We proceed by constructing a sequence $\{W_i\}_{i \geq 0}$ of nested sets satisfying $g(C + W_{i}) = \infty$. By Lemma \ref{lem: filling negatives}, there exists a set $D \subseteq \Z_{<0}$ such that $C + D \supseteq \Z_{\leq 0}$ and $g(C + D) = \infty$.  Set $W_0 = D$.  Let $z_i \in \nnn$ be the minimum element of $\zzz \cut (C + W_{i-1}) \subseteq \nnn$.    Set $k_i = g(C \cap [c_1,c_{i+1}])$, and observe that $[1,k_i] +[c_1,c_{i+1}) = [c_1 + 1, c_{i+1} + k_i)$.  We define 
$$W_{i} = W_{i-1} \cup \{z_i - c_i\} \cup ([1, k_i] + z_i - c_1).$$

Thus
\begin{align*}
    C + W_i &= (C + W_{i-1}) \cup (C + z_i  - c_i) \cup ((C + [1, k_i]) + z_i - c_1)\\
    &= (-\infty, z_i) \cup (C + z_i  - c_i) \cup ([c_1 + 1, c_{i+1} + k_i) + z_i - c_1)\\
    &\supseteq (-\infty, z_i) \cup \{z_i\} \cup ([c_1 + 1, c_{i+1} + k_i) + z_i - c_1)\\
    &= (-\infty, c_{i+1} - c_1 + k_i + z_i).
\end{align*}
Since $z_i \not\in (C + W_{i-1}) \cup ((C + [1,k_i]) + z_i - c_1)$, then $z_i \not\in W_i + (C \cut \{c_i\})$. Moreover, we have 
\begin{align*}
z_{i+1} &\geq \min\{C + c_{i+1} - c_1 + k_i + z_i\} = c_{i+1} + k_i + z_i > c_{i+1} + z_i\,.
\end{align*}
Thus $C + (W_{i+1} \cut W_i) \subseteq (z_i, \infty)$.  Since the $z_i$ are non-decreasing, we have $z_i \not\in W_k + (C \cut \{c_i\})$ for any $k \geq i$. By definition, $z_i \not\in W_k + (C \cut \{c_i\})$ for $k < i$.  Using the fact that $z_i$ is represented uniquely as $(z_i - c_i) + c_i$ in $C + W_i$, we have $z_i \notin W_k + (C \cut \{c_i\})$ for any $k \in \zzz_{\geq 0}$.   

Let $W = \bigcup_{i \geq 0} W_i$. Since $W_i \supseteq (-\infty, c_{i+1} - c_1 + k_i + z_i)$, we have $C + W = \zzz$.  Moreover, $z_i$ is not an element $W + (C \cut \{c_i\})$ for any $i \in \nnn$. Thus $C$ is a minimal additive complement to $W$.

We lastly check that every integer is finitely $(C,W)$-represented.  By Lemma \ref{lem: filling negatives}, every element of $C + D$ is finitely $(C,D)$-represented, so it is enough to show that every element of $C + (W\cut D)$ is finitely $(C, W \cut D)$-represented.  Recall that $D = W_0$, so $W \cut D = \bigcup_{i=0}^\infty (W_{i+1} \cut W_i)$.  As $W_{i+1} \cut W_i$ is finite for $i \geq 1$, it is sufficient to show that no integer appears in $C + (W_{i+1} \cut W_i)$ for infinitely many $i$.  This is clear from our previous observation that $C + (W_{i+1} \cut W_i) \subseteq (z_i, \infty) \subseteq (i,\infty)$.
\end{proof}

In order to complete our proof of Theorem \ref{thm: co-minimal strengthening}, we make use of Lemma \ref{lem: finite rep} from the work of Biswas and Saha \cite[Lemma 3.3]{BS2}.  The proof of Theorem \ref{thm: co-minimal strengthening} is then a combination of this lemma and the previous result.

\begin{proof}[Proof of Theorem \ref{thm: co-minimal strengthening}]
Let $C = \{c_1 < c_2 < \cdots\}$ be an infinite set of natural numbers with $g(C) = \infty$.  By Lemma \ref{lem: large gaps MAC}, $C$ arises as a minimal additive complement to a set $W$ such that every integer is finitely $(C,W)$-represented.  By Lemma \ref{lem: finite rep}, there exists a subset $W'$ of $W$ such that $W'$ is a minimal additive complement to $C$.  As $W' \subseteq W$, $C$ is a minimal additive complement to $W'$.  Thus $(C,W')$ is a co-minimal pair in the integers.
\end{proof}

We now shift our focus to certain sets containing arbitrarily large intervals in order to prove Theorem \ref{thm: long runs not MAC}.  

\begin{lemma}\label{lem: not MAC to finite}
If $C$ is a set of integers which 
satisfies $g(\Z \cut  C) = \infty$, then $C$ does not arise as a minimal additive complement to any finite set.
\end{lemma}
\begin{proof}
Suppose that $C$ is a minimal additive complement to a finite set $S = \{s_1 <  \cdots < s_n\}$.  Note that  $\Z \not= C$ implies $n \geq 2$.  Let $c$ and $c'$ be consecutive elements of $\Z \cut C$ satisfying that $c'-c > g(S) + 2 + (s_n - s_1)$. Consider the element $z = c + g(S) + 1 \in C$.  For $1 \leq i < n$, we have $z + s_i \in C\cut\{z\} + s_{i+1}$, since
$$    C\cut \{z\} + s_{i+1} - s_i \supseteq (c,z) + s_{i+1} - s_i \supseteq (c+ g(S), z + 1) = \{z\}\,.$$
Since
$$s_1 + c < z + s_n = g(S) + 1 + s_n + c < (c' - c) + s_1 + c = s_1 + c',$$
we have that $z + s_n \in (c,c') + s_1 \subseteq C + s_1$.  Moreover, $n \geq 2$ implies that $z + s_n \not= z + s_1$.  Hence $z + s_n \in C \cut \{z\} + s_1$.  

Therefore $z + S \subseteq C \cut\{z\} + S$.  Combining this with the assumption that $C + S = \Z$, we have that $C \cut\{z\} + S = \Z$, contradicting that $C$ is a minimal additive complement to $S$.
\end{proof}

\begin{remark}
Lemma \ref{lem: not MAC to finite} actually follows immediately from Theorem \ref{thm: MAC to finite density} along with the observation that a set satisfying the conditions of Lemma \ref{lem: not MAC to finite} has upper Banach density $1$.  However, we have provided the above proof of this lemma both to avoid a dependence on our later work in Section \ref{sec: finite} and because the argument is distinct from that used for Theorem \ref{thm: MAC to finite density}.
\end{remark}

\begin{proof}[Proof of Theorem \ref{thm: long runs not MAC}]
Suppose that $C$ is a minimal additive complement to some set $W$.  By Lemma \ref{lem: not MAC to finite}, $W$ must be infinite.  Fix distinct elements $w,w' \in W$ such that $w - w' > \ell$.  For sufficiently large $z \in \zzz$, $z \notin C$ implies $z + (w - w') \in C$.  Thus the set $\zzz \cut (C + \{w,w'\})$ is bounded above.  

Moreover, our assumption on $\overline{C_+}$ implies $g(\overline{C_+}) = \infty$.  Hence, there are infinitely many elements $c \in C_+$ such that both $c + (w - w'), c + (w' - w) \in C_+$.  Thus, neither $c + w$ nor $c + w'$ is dependent on $c$, since they can be represented twice in the sumset $C + W$ as $c + w = c + (w - w') + w'$ and $c + w' = c + (w - w') + w'$, respectively.  So some element of $\zzz \cut (C + \{w,w'\})$ must be dependent on $c$.  Since this holds for infinitely many positive $c \in C_+$, $W$ must contain infinitely many negative elements.    

Choose integers $z, z' \in \zzz$ such that $z,z'$ are dependent on $c, c' \in C_+$, respectively, and satisfy $z - z' > \ell$.  Thus, there are infinitely many negative $w \in W$ such that $z, z' \not\in C \cut \{c,c'\} + w$.  That is, $z-w, z'-w \not \in C \cut \{c,c'\}$ for infinitely many negative $w$. Hence there are infinitely many elements of $\overline{C_+}$ of distance $z' - z$ apart.  This contradicts our assumption that we have $\overline{c_{j}} - \overline{c_i} \notin [\ell,z' - z]$ for sufficiently large $i$ and $j > i$.
\end{proof}

A combination of the above results yields the proof of Theorem \ref{thm: dualCY}. 

\begin{proof}[Proof of Theorem \ref{thm: dualCY}]
The first claim follows directly from Lemma \ref{lem: large gaps MAC}.  The second claim is a special case of Theorem \ref{thm: long runs not MAC}, where $C_+ = C$ and $\ell = 1$.
\end{proof}

\begin{observation}\label{obs: has, isn't}
Let $C$ be a set satisfying Condition $(ii)$ of Theorem \ref{thm: dualCY}, and let $B$ be any infinite, bounded-above set.  Note that Theorem \ref{thm: long runs not MAC} implies that $B \cup C$ does not arise as a minimal additive complement, while Theorem \ref{thm: CY1} implies that $B \cup C$ has a minimal additive complement.  Thus, we have a class of sets that have but do not arise as minimal additive complements.

The sets satisfying Condition $(i)$ of Theorems \ref{thm: CY2} and \ref{thm: dualCY} form a class of sets that both have and arise as minimal additive complements.  Similarly, the sets satisfying Condition $(ii)$ of Theorems \ref{thm: CY2} and \ref{thm: dualCY} form a class of sets that do not have and do not arise as minimal additive complements.  It remains to exhibit an infinite class of sets that arise as but do not have minimal additive complements.  We will see such a class in Section \ref{sec: eventually periodic}, namely those sets discussed in Remark \ref{rem: arise but don't have MAC}.
\end{observation}

\begin{remark}
Note that, as is the case in Chen and Yang's result (Theorem \ref{thm: CY2}),  the limit in the Condition (ii) of Theorem \ref{thm: dualCY} cannot be improved to a limit superior.  For example, consider the set
$$W = \{w_1 < w_2 < \cdots\} =\bigcup_{k = 0}^\infty \left[2^{2k}, 2^{2k+1}\right)\,.$$
Let $\nnn \cut W = \{\overline{w_1} < \overline{w_2} < \cdots\}$.  Then $W$ satisfies $\limsup\limits_{i \to \infty} w_{i+1} - w_i = \infty$ and $\limsup\limits_{i \to \infty} \overline{w_{i + 1}} - \overline{w_i} = \infty$.
\end{remark}

\section{Syndetic Sets}\label{sec: finite}
We now shift our focus to syndetic sets of integers arising as minimal additive complements.  A set $A \subseteq \Z$ is called {\it syndetic} if it is bounded neither above nor below and has bounded gaps, i.e., if there exists $k\in \N$ such that $A \cap I$ is nonempty for every interval $I$ of length $k$.  It it straightforward to see that the syndetic sets are precisely the sets that are (not necessarily minimal) additive complements to finite sets in the integers.  
In this section, we bound the upper and lower Banach density of syndetic sets arising as minimal additive complements to finite sets.

In particular, these results recover some of the density bounds on periodic sets arising as minimal additive complements in the integers.  A set of integers $S\subseteq Z$ is said to be \textit{periodic} with period $m\in \N$ provided that $S + m =S$ for some $m \in \N$ and $m$ is the minimum positive integer with this property.  It is straightforward to see that if $S \subseteq \Z$ is a periodic set with period $m$, then $S$ arises as a minimal additive complement in $\Z$ if and only if $S_{(m)}$ arises as a minimal additive complement in $\Z/m\Z$.  This follows from a more general result about quotients in abelian groups, presented in Proposition \ref{prop: quotient MACS}.

Let $H\subseteq G$ be a subgroup of $G$. A set $S\subseteq G$ is said to be $H$\textit{-invariant} provided that 
$H + S = S$, i.e., $S$ is a union of cosets of $H$.
Suppose $S\subseteq G$ is $H$-invariant. We will use the notation
\[
S/H := \pi(S),
\]
where $\pi\colon G\to G/H$ is the projection map.

\begin{proposition}\label{prop: quotient MACS}
Let $G$ be an abelian group and let $H\subseteq G$ be a subgroup. Suppose $C\subseteq G$ is $H$-invariant. Then $C$ arises as a minimal additive complement in $G$ if and only if $C/H$ arises as a minimal additive complement in $G/H$.

Moreover, in this case, $C$ and $C/H$ arise as minimal additive complements to $W\subseteq G$ and $S\subseteq G/H$, respectively, with $|W| = |S|$.
\end{proposition}
\begin{proof}
Suppose $C$ is a minimal additive complement to $W$ in $G$.   Then we have $$C/H + W/H = (C + W)/H = G/H\,,$$
so $C/H$ is an additive complement to $W/H$ in $G/H$.  Suppose there exists a proper subset $S\subset C/H$ such that $S + W/H = G/H$. Let $D=\pi^{-1}(S)$, where $\pi\colon G\to G/H$ is the projection map.
Note that $D$ is $H$-invariant, and moreover that $D$ is a proper subset of $C$ because $D/H$ is a proper subset of $C/H$ and $C$ is $H$-invariant. Furthermore $D/H + W/H =S + W/H= G/H$. 
That $D$ is an additive complement to $W$ in $G$ follows from the fact that $D/H$ is an additive complement to $W/H$ in $G/H$ and both $D$ and $W$ are $H$-invariant.  This yields a contradiction, as $D$ is a proper subset of $C$.

Conversely, suppose $C/H$ is a minimal additive complement to $S\subseteq G/H$. Construct $W\subseteq G$ such that it contains exactly one element of $\pi^{-1}(s)$ for each $s\in S$, where $\pi\colon G\to G/H$ is the projection map. It is straightforward to check that $C$ is a minimal additive complement to $W$. Finally, when $W$ is constructed in this way, we have $|W|=|S|$.
\end{proof}


Alon, Kravitz, and Larson \cite{AKL} recently studied the size of sets arising as minimal additive complements in finite groups.  As they noted, by Proposition \ref{prop: quotient MACS} their results on finite cyclic groups extend to periodic sets in the integers and show that almost all periodic sets (when ordered by period) do not arise as minimal additive complements.

We briefly mention an extension of their upper bound on the density of minimal additive complements in finite groups \cite[Proposition 13]{AKL}, which we require later for the proof of Corollary \ref{cor: ev per density}.  In Lemma \ref{lem: UBminimalsetonly}, we show that the argument extends for sets $C$ in a finite group that are minimal with respect to having the sumset $C + W$ for some set $W$.  Of course, both $G$ and $\{0\}$ arise as minimal additive complements in any group, but excluding these cases there are nontrivial restrictions on the density.

\begin{lemma}\label{lem: UBminimalsetonly}
Suppose $C,W\subseteq G$ and for every proper subset $D\subset C$, $D + W \neq C + W$. Let $|W|= k > 1$. Then
$$    |C| \le \frac{k}{2k-1}|G|\,.$$
\end{lemma}
\begin{proof}
For each $c \in C$, choose an element $$z(c) \in (C + W) \cut ((C \cut \{c\}) + W)\,.$$  Note that such an element exists by our assumption of the minimality of $C$ with respect to the sum $C + W$.  Let $Z = \{z(c) : c \in C\} \subseteq C + W$.  Observe that for each $z \in Z$, there is a unique element $w(z)$ of $W$ such that $z \in C + w(z)$.  In particular, if $z = z(c)$ then $w(z) = z - c$.  Let $W= \{w(z) : z \in Z\}$. 

Let $M$ be the set of pairs $(z,w')$ where $z \in Z$ and $w'\in W\setminus \{w(z)\}$. Since $|Z| = |C|$, there are $|C|\cdot(|W| - 1) = |C|\cdot (k-1)$ such pairs in $M$. Let $N$ be the set of pairs $(z-w',w')$ with $z \in Z$ and $w'\in W\setminus \{w(z)\}$. Observe that if $(z-w',w') \in N$, then $z - w' \in G \cut C$.  This follows because if $z - w'$ were in $C$, then $z \in C + w'$. Hence $w' = w(z)$, contradicting our choice of $w'$.  Thus $N \subseteq (G \cut C)\times W$, so $|N| \leq (|G| - |C|)\cdot k$.  

The set of pairs $N$ is in bijection with the set of pairs $M$ via the map $(z,w')\mapsto (z-w', w')$. Therefore
\[
    (k-1)|C| = |M| = |N| \leq (|G|-|C|)\cdot k\,,
\]
which directly implies the stated inequality.
\end{proof}

We now introduce the notion of Banach density in order to extend these combinatorial bounds to syndetic sets that are not necessarily periodic.  Note that these limits always exist; see \cite{GTT} for a proof and alternate definitions. 

\begin{definition}
The {\it upper Banach} (or {\it uniform}) {\it density} of $C \subseteq \Z$ is given by
$$\overline{b}(C) = \lim_{n \to \infty} \max_{x \in \Z} \frac{|A \cap [x+1, x+n]|}{n}\,.$$
The {\it lower Banach} (or {\it uniform}) {\it density} of $C \subseteq \Z$ is given by
$$\underline{b}(C) = \lim_{n \to \infty} \min_{x \in \Z} \frac{|A \cap [x+1, x+n]|}{n}\,.$$
\end{definition}

Proposition \ref{prop: quotient MACS} combined with Lemma \ref{lem: UBminimalsetonly} recovers the result of Alon, Kravitz, and Larson \cite[Proposition 13]{AKL} stating that periodic sets arising as minimal additive complements in $\Z$, aside from $\Z$ itself, have upper Banach density at most $2/3$.  In Theorem \ref{thm: MAC to finite density}, we show the same result holds for all syndetic sets
which arise as minimal additive complements to finite sets.
We also give a straightforward argument that the lower Banach density of a syndetic set $W$ is at least the reciprocal of the size of the smallest set to which $W$ is as a minimal additive complement. By Proposition \ref{prop: quotient MACS}, these bounds also immediately apply to all periodic sets.

\begin{theorem}\label{thm: MAC to finite density}
Suppose that $C \subseteq \Z$ arises as a minimal additive complement to a set of size $k \geq 2$.  Then
$$\frac{1}{k} \leq \underline{b}(C) \leq \overline{b}(C) \leq \frac{k}{2k - 1}\,.$$
\end{theorem}
\begin{proof}
Let $C$ arise as a minimal additive complement to $W = \{0 = w_1<w_2<\cdots <w_k\}$.  

We begin by proving the lower bound.  For $x \in \Z$ and $n \in \N$ , we have
\begin{align*}
\left|[x + 1,x + n]\cap (C + W)\right| 
&\le \sum_{i=1}^k \left|[x + 1 - w_i,x + n-w_i]\cap C\right|\\
&\leq k\cdot |[x + 1 - w_k, x + n]\cap C|\\
&\leq k\cdot |[x + 1, x + n]\cap C|\ + k(w_k+1)\,.
\end{align*}

Dividing the first and last quantities by $n$, we see that the term $k(w_k+1)/n$ vanishes as $n$ approaches infinity.  Thus we have $\underline{b}(C + W) \leq k\underline{b}(C)$.   As we assume $C + W = \Z$, which has lower Banach density $1$, this implies $\underline{b}(C) \geq \frac{1}{k}$.

Now to address the upper bound. Consider the set of pairs $(w, z - w)$ where $z$ is dependent on some $c_z \in [x+1,x+n] \cap C$ and $w\in W\setminus \{z - c_z\}$. As there are at least $|[x+ 1, x + n]\cap C|$ integers that are dependent on some $c \in [x+ 1, x + n]\cap C$, there are at least $(k-1)|[x+ 1, x + n]\cap C|$ such pairs.  Moreover, suppose that $z - w \in C$.  Then we would have $(z - w) + w \in (C\cut{c_z}) + W$, contradicting that $z$ is dependent on $c_z$.  Thus, $z - w \in [x + 1 - w_k, x + n] \cut C$, so there are at most $k( n + w_k - |[x + 1 - w_k, x + n]\cap C|)$ such pairs.  Hence
\begin{align*}
(k-1)|[x+ 1, x + n]\cap C| &\leq k( n + w_k - |[x + 1 - w_k, x + n]\cap C|)\\ 
&\leq k(n - |[x+1,x+n] \cap C|) + kw_k\,,
\end{align*}
which implies
$$\frac{[x+1,x+n] \cap C}{n} \leq \frac{k}{2k-1} + \frac{kw_k}{(k-1)n}\,.$$
As $n$ approaches infinity, the last term vanishes.  Since this inequality holds for all choices of $x \in \Z$, we can conclude $\overline{b}(C) \leq \frac{k}{2k-1}$.
\end{proof}

\section{Eventually Periodic Sets}\label{sec: eventually periodic}
In their study of sets of integers having minimal additive complements, Kiss, S\'andor, and Yang \cite{KSY}  introduced the notion of \emph{eventually periodic sets} of integers.  

\begin{definition}\label{def: eventually periodic}
An \emph{eventually periodic set} (\emph{of period $m$}) is a set of integers of the form 
$$ (m\nnn + A) \cup B \cup F\,,$$
where $A$ is nonempty, $B$ and $F$ are finite, $B_{(m)} \subset A_{(m)}$, and $F_{(m)} \cap A_{(m)} = \emptyset$.  
\end{definition}

In other words, a set of integers is eventually periodic if it consists of a finite union of infinite (to the right) arithmetic progressions and singletons. Note that if we furthermore require that $m$ is the minimal period, $b + m \notin A$ for any $b \in B$, and that $A$ has at most one element in each equivalence class modulo $m$, then the representation of an eventually periodic set in terms of $(m,A,B,F)$ is unique.  We assume this is the case in the following discussion.  Kiss, S\'andor, and Yang showed that, in many cases, the question of whether an eventually periodic set $(m\nnn + A) \cup B \cup F$ has a minimal additive complement can be reduced to conditions on $A_{(m)}$ and $F_{(m)}$ in $\zzz/m\zzz$. However, in the dual setting such a reduction is not always possible, even in the case $m = 2$. This is illustrated by the following proposition, proved in greater generality later within this section.

\begin{proposition}\label{prop: evens odds}
Let $C = 2\nnn \cup B \cup F$, where $B \subset 2\zzz_{\leq 0}$ and $F \subset 2\zzz + 1$ are finite sets.  Then $C$ arises as a minimal additive complement if and only if $2\zzz \cut (2\nnn \cup B) = W + F$ for some $W \subseteq \zzz$
\end{proposition}

Note that the latter condition can be checked in finitely many steps, as $B$ and $F$ are finite.  However, this proposition illustrates that the information of $A_{(m)}, B_{(m)}, F_{(m)} \subseteq \Z/m\Z$ alone does not determine whether $(m\N + A) \cup B \cup F$ arises as a minimal additive complement.  Throughout this section we derive necessary conditions (which can be tested in finitely many steps) for eventually periodic sets to arise as minimal additive complements, and furthermore we construct a class of sets for which these conditions are also sufficient.

\begin{lemma}\label{lem: single equiv class}
Let $C \supset m\nnn + A$, where $|A \cap (z + m\zzz)| \leq 1$ for each $z \in [0,m)$.  If $C$ is a minimal additive to $W \subseteq \zzz$, then for each $a \in A$ there exists $w_a \in W$ such that $a + w_a \not\in A + (W \cut \{w_a\}) + m\zzz$.  
\end{lemma}
\begin{proof}
Suppose $C$ is a minimal additive complement to $W$, and fix $z \in [0,m)$. Observe that if there are infinitely many elements in $(z + m\zzz_{<0}) \cap W$, then no element of $z + m\zzz$ is dependent on any element of $C$.  If $(z + m\zzz) \cap W$ is bounded below and of size at least $2$, let $w_1$ and $w_2$ be the two smallest elements.  Then, for any $a \in A$, at most $|w_2 - w_1|$ elements of $a + z + m\zzz$ are dependent on $a + m\nnn \subseteq C$.  Similarly, if there exists $w,w' \in W$ and distinct $a, a' \in A$ such that $w + a, w' + a' \in z + m\zzz$, then at most $|(w + a) - (w' + a')|$ elements of $z + m\zzz$ are dependent on $a + m\nnn \subseteq C$.  Since, for each $a \in A$, there are infinitely many integers are dependent on $a + m\nnn \subseteq C$, there must exist $w_a \in W$ such that $(w_a + m\zzz) \cap W = \{w_a\}$ and no element of $(A\cut \{a\}) + (W \cut \{w_a\})$ is equivalent to $a + w_a$ modulo $m$.  Thus, $w_a$ satisfies that $a + w_a \not\in A + (W \cut \{w_a\}) + m\zzz$.
\end{proof}

\begin{theorem}\label{thm: ep necc}
Let $C = (m\nnn + A) \cup B \cup F$ be as in Def.~\ref{def: eventually periodic}.  If $C$ arises as a minimal additive complement, then there exists $ \overline{Y} = Y_+ \cup Y_- \cup Y_0 \subseteq \zzz/m\zzz$ such that
\begin{enumerate}
\item $\left(A_{(m)} + \overline{Y}\right) \cup \left(F_{(m)} + Y_+\right)  = (A_{(m)} \cup F_{(m)}) + Y_- = \zzz/m\zzz$;
\item for each $a \in A_{(m)}$, there exists $y_a \in Y_0$ such that 
$$a + y_a \in  (F_{(m)} + Y_-) \cut ((A_{(m)}\cut \{a\} + Y_0) \cup (A_{(m)} + (Y_- \cup Y_+))\,;$$
\item for each $f \in F_{(m)}$, there exists $y \in \overline{Y}$ such that $f + y \not\in A_{(m)} + Y_-$.
\end{enumerate}
\end{theorem}
\begin{proof}
Suppose that $C$ arises as a minimal additive complement to $Y \subseteq \zzz$.  Define
\begin{align*}
    Y_+ &= \{z \in \zzz/m\zzz : (m\nnn + z) \cap Y \text{ is infinite}\}\,,\\
    Y_- &= \{z \in \zzz/m\zzz : (m\zzz_{<0} + z) \cap Y \text{ is infinite}\}\,, \text{ and}\\
    Y_0 &= \{z \in \zzz/m\zzz : (m\zzz + z) \cap Y \text{ is nonempty and finite}\}\,.
\end{align*}
Observe that 
$$\{z \in \zzz/m\zzz: (m\nnn + z) \cap (C + Y) \text{ is infinite}\}= \left(A_{(m)} + \overline{Y}\right) \cup \left(F_{(m)} + Y_+\right)\,.$$    Similarly,  the set $\{z \in \zzz/m\zzz: (m\zzz_{<0} + z) \cap (C + Y) \text{ is infinite}\}$ is equal to $(A_{(n)} \cup F_{(n)}) + Y_-$.  Since $C$ is an additive complement to $Y$, these sets are all equal to $\zzz/m\zzz$.  Thus, Condition (1) is satisfied.

Fix $a \in A$. By Lemma \ref{lem: single equiv class}, there exists $y \in Y$ such that $a + y \not\in A + Y \cut \{y_a\} + m\Z$.  Letting $y_a = \{y\}_{(m)}$, this implies $y_a \not\in((A_{(m)}\cut \{a\} + Y_0) \cup (A_{(m)} + (Y_- \cup Y_+))$.  By Condition (1), we must have $y_a \in F_{(m)} + Y_-$. Therefore $y \in (F_{(m)} + Y_-) \cut ((A_{(m)}\cut \{a\} + Y_0) \cup (A_{(m)} + (Y_- \cup Y_+))$.

Fix $f \in F$ and $y' \in Y$ such that the integer $f + y'$ is dependent on $f$.  Observe that no element of $A_{(m)} + Y_- + m\zzz$ is dependent on any element of $C$, in particular on any element of $F$.  Setting $y = \{y'\}_{(m)}$, we can conclude that $f + y \not\in A_{(m)} + Y_-$.
\end{proof}

Combining this with the density bound of Theorem \ref{lem: UBminimalsetonly} from the previous section, we obtain the following corollary.

\begin{corollary}\label{cor: ev per density}
Let $C = (m\nnn + A) \cup B \cup F$ be as in Def.~\ref{def: eventually periodic}.  If $C$ arises as a minimal additive complement, then $2|A_{(m)}| \leq m + |F_{(m)}|$.
\end{corollary}
\begin{proof}
Let $C$ be a minimal additive complement to $Y \subseteq \zzz$, and define $Y_+$, $Y_-$, and $Y_0$ as in the proof of Theorem \ref{thm: ep necc}.  Condition (2) of Theorem \ref{thm: ep necc} implies that $A_{(m)}, \overline{Y} \subseteq \Z/m\Z$ satisfy the relation $(A_{(m)} \cut \{a\}) + \overline{Y} \neq A_{(m)} + \overline{Y}$ for each $a \in A_{(m)}$.  Applying Lemma \ref{lem: UBminimalsetonly}, we have 
\begin{equation}\label{eq: density}
    |A_{(m)}| \leq \frac{|\overline{Y}|}{2|\overline{Y}| - 1}m\,.
\end{equation}
Moreover, Condition (2) of Theorem \ref{thm: ep necc} implies that $|F_{(m)}|\cdot |\overline{Y}| \geq |F_{(m)} + Y_-| \geq |A_{(m)}|$, hence $|\overline{Y}| \geq |A_{(m)}|/|F_{(m)}|$.  Since $|\overline{Y}|/(2|\overline{Y}| - 1)$ is a decreasing function in $|\overline{Y}| \in \N$, then substituting $|A_{(m)}|/|F_{(m)}|$ for $|\overline{Y}|$ in \eqref{eq: density} yields the desired inequality.
\end{proof}

Using a similar approach of bounding the relative sizes of $|A_{(m)}|$ and $|F_{(m)}|$, we can derive a stronger bound when $m$ is prime.

\begin{theorem}\label{thm: add necc}
Let $C = (m\nnn + A) \cup B \cup F$ be as in Def.~\ref{def: eventually periodic}.  If $m$ is prime and $C$ arises as a minimal additive complement, then 
$$|A_{(m)}| \leq \frac{|F_{(m)}|}{2|F_{(m)}| + 1} (m+1)\,.$$

Moreover, if $|A_{(m)}| > \frac{|F_{(m)}|}{2|F_{(m)}| + 1}m$, then 
\begin{enumerate}
    \item[(4)] for each $a \in A_{(m)}$, $(m\Z + a) \cap (\Z \cut C) = (m \Z + a) \cap (F + W)$ for some $W \subseteq \Z$. 
\end{enumerate} 
\end{theorem}
\begin{proof}
We begin by proving the upper bound on $|A_{(m)}|$.  Suppose $C$ arises as a minimal additive complement to $Y \subseteq \Z$, and define $Y_+$, $Y_-$, and $Y_0$ as in the proof of Theorem \ref{thm: ep necc}.  By the Cauchy-Davenport theorem, we have $|A_{(m)} + Y_-| \geq \min\{|A_{(m)}| + |Y_-| - 1, m\}$.  Note that no element of $\Z$ is dependent on any element of $(A_{(m)} + Y_-) + m\Z$. Thus we cannot have $|A_{(m)}| + |Y_-| - 1 > m$, otherwise $A_{(m)} + Y_- = \Z/m\Z$.  So we in fact have $|A_{(m)} + Y_-| \geq |A_{(m)}| + |Y_-| - 1$.

By Condition (2) of Theorem \ref{thm: ep necc}, for each $a \in A_{(m)}$, there exists an element $y_a \in Y_0$ such that $a + y_a \notin (A_{(m)} + Y_-) \cup (A_{(m)} \cut \{a\} + Y)$.  Moreover, the same result implies that we  have $|F_{(m)} + Y_-| \geq |A_{(m)}|$, hence $|F_{(m)}||Y_-| \geq |A_{(m)}|$.  Combining these inequalities, we have
$$|A_{(m)} + Y_-| \geq |A_{(m)}| + |Y_-| - 1 \geq |A_{(m)}| + |A_{(m)}|/|F_{(m)}| - 1$$
Since the sum $a + y_a$ is distinct for each $a \in A_{(m)}$ and is not contained in $A_{(m)} + Y_-$, we have $m \geq |A_{(m)}| + |A_{(m)} + Y_-|$, yielding
$$|A_{(m)}| \leq \frac{|F_{(m)}|}{2|F_{(m)}| + 1} (m+1)\,.$$

On the other hand, if $|A_{(m)}| > \frac{|F_{(m)}|}{2|F_{(m)}| + 1}m$, then we must have $m = |A_{(m)}| + |A_{(m)} + Y_-|$.  Thus for each $a \in A_{(m)}$, there is a unique $y_a \in Y_0$ such that $a + y_a \notin A_{(m)} + Y_-$.  In particular, every element of $C \cap (m\Z + a)$ must have a dependent element in $m\Z + a + y_a$.  So there is a unique element $y \in Y \cap (y_a + m\Z)$, and every element of $(C + y) \cap (m\Z + a)$ is covered exactly once in $C + Y$, in particular by $C + y$.  Therefore, there must exist $W \subseteq \Z$ such that 
$$(m\Z + a + y) \cap (\Z \cut C + y) = (m\Z + a + y) \cap (C \cut (C \cap (m\Z + a)) + W)\,.$$
Since there are finitely many positive elements of $(m\Z + a + y) \cap (\Z \cut C + y)$, then in fact, translating by $y$, we have
$$(m\Z + a) \cap (\Z \cut C) = (m\Z + a) \cap (F + W)\,.\text{\qedhere}$$
\end{proof}

We now shift our focus to showing that Conditions (1)-(3) of Theorem \ref{thm: ep necc} as well as Condition (4) of Theorem \ref{thm: add necc} are sufficient for certain eventually periodic sets to arise as minimal additive complements.  We first consider a lemma whose proof is analogous to that of Lemma \ref{lem: finite rep}. 

\begin{lemma}\label{lem: finite rest cover}
Fix a finite set $F \subset \zzz$ and $W \subseteq \Z$ such that $F + W \supset \N$.  Then there exists a set $W' \subseteq \Z$ such that $F + W' = F + W$ and $F' + W' \neq F + W$ for any proper $F' \subset F$.  
\end{lemma}
\begin{proof}
Let $F = \{f_1 < \cdots < f_k\}$, and $F_i = F \cut \{f_i\}$.  By performing an appropriate translation on $F$ and $W$, we can assume $f_1 = 0$.  Let 
$$W' = (W \cup \nnn) \cut \left( \bigcup_{i = 1}^k (-F_i + 2i \cdot f_k) \right)\,.$$
Observe that $2i \cdot f_k - f_j$ is contained in $W'$ if and only if $i = j$.  Thus $2i \cdot f_k \notin F\cut\{f_i\} + W'$.  So $F' + W' \not= F + W'$ for any proper subset $F' \subset F$.  

Secondly, we show that $F + W' = F + W$.  Since we assumed $f_1 = 0$ and $F + W \supseteq \N$, we have $F + W' \subseteq F + (W \cup \N) \subseteq F + W$.  For the other inclusion, we have 
$$F + W' \supseteq (F + W) \cup (F + \N) =  (F + W) \cup \N = F + W\,.\text{\qedhere}$$
\end{proof}

\begin{remark}\label{rem: arise but don't have MAC}
This lemma is a useful tool in showing that certain eventually periodic sets arise as minimal additive complements, as we shall see in the following theorem. 

One particular example includes sets of the form $3\N \cup F$ where $F = \{f_1 < \dots < f_n\}$ is a finite subset of $3\N + 1$.  Note that $\{0\}  \cup (3\Z + 1)\cup (-3\Z_{\geq 0} - f_n)$ is a complement to $3\N \cup F$.  Moreover, note that $F + (-3\Z_{\geq 0}  - f_n) = -3\Z_{\geq 0}$, so an application of Lemma \ref{lem: finite rest cover} (along with an appropriate scaling and shift) yields a set $W' \subseteq 3\Z + 2$ such that $W' + F = -3\Z_{\geq 0}$ and $W' + F' \neq -3\Z_{\geq 0}$.  It is then straightforward to check that $3\N + F$ arises as a minimal additive complement to  $\{0\} \cup (3\Z + 1)\cup W'$.  

However, no set of the form $3\N \cup F$ where $F \subseteq 3\N + 1$ has as minimal additive complement by Corollary 1 of \cite{KSY}.  Thus, this constitutes a class of sets which arise as, but do not have, minimal additive complements.
\end{remark}

\begin{theorem}\label{thm: ep suff}
Let $C = (m\nnn + A) \cup B \cup F$ be as in Def.~\ref{def: eventually periodic}.  Suppose $m$ is a prime which is equivalent to $2$ modulo $3$, $|F| = 1$, and $|A| = (m+1)/3$.  If Conditions (1) - (3) of Theorem \ref{thm: ep necc} and Condition (4) of Theorem \ref{thm: add necc} all hold, then $C$ arises as a minimal additive complement.
\end{theorem}
\begin{proof}
Let $\overline{Y} = Y_+ \cup Y_- \cup Y_0$ be as in Theorem \ref{thm: ep necc}.  By Condition (2) of Theorem \ref{thm: ep necc}, for each $a \in A_{(m)}$ there exists $y_a \in Y_0$ such that 
$$a + y_a \in (F_{(m)} + Y_-) \cut ((A_{(m)}\cut \{a\} + Y_0) \cup (A_{(m)} + (Y_- \cup Y_+))\,.$$  
As in the proof of Theorem \ref{thm: add necc}, for each $a \in A_{(m)}$ the choice of $y_a$ is unique.  Moreover, the set $\Z \cut \left(\bigcup_{a \in A}a + y_a + m\Z\right)$ is contained in $A_{(m)} + Y_- + m\Z$.  

Let $F_{(m)} = \{f\}$, where $f \in \Z/m\Z$.  Thus $a + y_a - f \in Y_-$.  By Condition (4) of Theorem \ref{thm: add necc}, there exists $W_a \subseteq m\Z + (a - f)$ such that $m\Z + (a + y_a) \cap (\Z \cut C) = m\Z + (a + y_a) \cap (F + W_a)$.  Moreover, we can apply Lemma \ref{lem: finite rest cover} to $(W_a - (a - f))/m = \{w/m : w \in W_a - (a-f)\}$ and $(F - f)/m = \{y/m : y \in F - f\}$.  Thus there exists $W_a' \subseteq \Z$ such that $F + W_a' = F + W_a$ and $F + W_a' \neq F' + W_a'$ for any proper subset $F'$ of $F$.  

Let
$$W = \left(\bigcup_{a \in A} \{y_a\} \right) \cup \left(\bigcup_{a \in A} W_a'\right)\,.$$

Fix $x \in \Z/m\Z$.  If $x \in a + y_a + m\Z$ for some $a \in A$, then 
$x + m\Z$ is the disjoint union of $(y_a + C)\cap (x + m\Z)$ and $F + W_a'$, hence $x + m\Z \subseteq C + W$. In particular, if $c \in C \cap (x - y_a + m\Z)$, then $c + y_a \not\in C \cut \{c\} + W$.  Similarly, by our assumption on the choice of $W_a'$, for any proper subset $F' \subset F$ we have 
$$F' + W_a' \not= F + W_a' = ((\Z\cut (y_a + C)) \cap (x + m\Z)$$
Hence $C \cut \{f\} + W \not= C + W$.  

It remains to show that $C$ and $W$ are complements.  As discussed in the previous paragraph, we have 
$$C + W \supseteq \displaystyle\bigcup_{a \in A} (a + y_a + m\Z)\,.$$
As noted in the proof of Theorem \ref{thm: add necc}, we have $A + Y_- = \Z/m\Z \cut \{a + y_a : a \in A_m\}$.  Hence
$$A + \bigcup_{a \in A} W_a = (\Z/m\Z \cut \{a + y_a : a \in A_m\}) + m\Z\,.$$
Therefore $C$ is a minimal additive complement to $W$.  
\end{proof}

Proposition \ref{prop: evens odds} now follows as a special case of Theorems \ref{thm: ep necc}, \ref{thm: add necc}, and \ref{thm: ep suff}.

\begin{proof}[Proof of Proposition \ref{prop: evens odds}]
Let $m = 2$, $Y_- = \{0\} \subset \Z/2\Z$, $Y_+ = \emptyset$, and $Y_0 = \{1\} \subset \Z/2\Z$.  Then observe that $2\N \cup B \cup F$ satisfies the conditions of Theorems \ref{thm: ep necc}, \ref{thm: add necc}, and \ref{thm: ep suff} (with respect to the choice of $\overline Y = Y_- \cup Y_0$) if and only if $F$ and $B$ are as claimed.
\end{proof}

\section{Further Directions}\label{sec: further}

As discussed in Section \ref{sec: aperiodic}, Theorem \ref{thm: long runs not MAC} is a strengthening of the second part of Theorem \ref{thm: dualCY}.  We conjecture that the dual statement to Theorem \ref{thm: long runs not MAC} holds for sets of natural numbers.  Observe that the second part of Theorem \ref{thm: CY2} corresponds to the case $\ell = 1$ of the following conjecture. 
\begin{conjecture}\label{cor: extendedCY2}
Let $C = \{c_1 < c_2 < \cdots\}$ be an infinite, bounded-below set of integers,  and let
$$\overline{C} = \nnn \cut C = \{\overline{c_1} < \overline{c_2} < \cdots\}\,.$$
If there exists a $\ell > 0$ such that for all $m \in \nnn$ we have $\overline{c_{j}} - \overline{c_i} \notin [\ell,m]$ for sufficiently large $i$ and $j > i$, then $C$ does not have a minimal additive complement.
\end{conjecture}

One could also study a variant of the eventually periodic sets, allowing for eventually periodic behavior in both the positive and negative directions.
\begin{question}
Which sets of the form $E_1 \cup E_2$, where $E_1$ and $-E_2$ are eventually periodic sets of integers, arise as minimal additive complements?
\end{question}
Note that such a set has a minimal additive complement by Theorem \ref{thm: CY1}.  We expect the behavior of these sets arising as minimal additive complements to be complex but interesting, and many of the methods used to analyze eventually periodic sets in Section \ref{sec: eventually periodic} may apply.  

A natural partition of the subsets of the integers is constructed as follows: $A \sim B$ if and only if the symmetric difference $A \Delta B$ is finite, i.e., $A$ and $B$ differ in finitely many places.  If a set $A$ arises (or does not arise) as a minimal additive complement, one can study whether sets $B$ such that $B \sim A$ also arise (or do not arise) as minimal additive complements.  

\begin{definition}
We say that a set $A\subseteq \Z$ is a \emph{robust MAC} if every $B \subseteq \Z$ satisfying $B \sim A$ arises as a minimal additive complement.  Similarly,  $A\subseteq \Z$ is a \emph{robust non-MAC} if every $B \subseteq \Z$ satisfying $B \sim A$ does not arise as a minimal additive complement.  
\end{definition}

We note that the sets satisfying Condition (i) of Theorem \ref{thm: dualCY} are robust MACs, and those satisfying Condition (ii) are robust non-MACs.  An eventually periodic set $(m\N + A) \cup B \cup F$ cannot be a robust MAC, since $m\N + A$ does not arise as a minimal additive complement.  However, eventually periodic sets may be robust non-MACs, e.g., the set $4\N + \{0,1,2\}$ by Corollary \ref{cor: ev per density}.  Kwon \cite{Kwo} showed that the finite sets are robust MACs.  It may be interesting to find further classes of robust MACs or non-MACs.

\begin{question}
Which sets are robust MACs or robust non-MACs? In particular, is there a syndetic robust MAC?
\end{question}

It may be interesting to determine which eventually periodic sets belong to a co-minimal pair, as some of the methods used in Section \ref{sec: eventually periodic} may apply.  Moreover, one could study the related question of whether there exists an eventually periodic set that both arises as and has a minimal additive complement but is not part of a co-minimal pair.  More generally, we raise the following question.

\begin{question}
Does there exist a single set of integers that both arises as and has a minimal additive complement but is not part of a co-minimal pair?
\end{question}

The study of minimal additive complements is also connected to the study of domination parameters of certain graphs, in particular the unitary Cayley graph of $\Z/n\Z$; see \cite{AD, Bur, DI, Vem} for some recent work on these graph parameters.  The \emph{unitary Cayley graph of $\mathbb{Z}/n\mathbb{Z}$}, denoted by $X_n$, is the graph on  $\{0,\dots,n-1\}$ (often viewed as $\Z/n\Z$) where vertices $a$ and $b$ are adjacent if and only if $\gcd(a-b,n) = 1$.  A set $W \subset \Z/n\Z$ is a \emph{dominating set} in $X_n$ if every vertex is in the closed neighborhood of $W$, and furthermore a dominating set is \emph{minimal} if none of its proper subsets is a dominating set.  Thus $W$ is a minimal dominating set of $X_n$ if and only if the set $W$ is a minimal additive complement to $P = (\{0\} \cup \{x \in \{1,\dots,n-1\} : \gcd(x,n) = 1\})$ in $\Z/n\Z$.  The maximum size of a minimal dominating set in $X_n$, known as the \emph{upper domination number} of $X_n$, is then precisely the maximum size of a minimal additive complement to $P + n\Z$ in the integers. Similarly, the minimum size of a (minimal) dominating set in $X_n$, known as the \emph{domination number} of $X_n$, is the minimum size of a (minimal) additive complement to $P + n\Z$ in the integers.  Similar statements can be given for any vertex-transitive graph.  This motivates a more general study of the sizes of minimal additive complements to a fixed syndetic set of integers.  
\begin{question}
Given a syndetic set $W \subseteq \Z$, what is the maximum and minimum size of a minimal additive complement to $W$ in the integers?
\end{question}

\section{Acknowledgments}\label{Sec:Acknowledgments}
We thank Andrew Kwon for a useful conversation leading to the consideration of robust MACs and non-MACS.  We also thank Max Kontorovich for providing us with a supportive work environment. The first author was supported by a Marshall scholarship and a St.\ John's College Benefactors' scholarship.

\end{document}